\documentclass[a4paper]{amsart}
\usepackage[pdftex]{graphicx}

\usepackage{amsmath, amsfonts, amssymb, amsthm}
\usepackage[all]{xy}

\newtheorem{theorem}{Theorem}[section]
\newtheorem{lemma}[theorem]{Lemma}
\newtheorem{proposition}[theorem]{Proposition}

\newtheorem{remark}[theorem]{Remark}

\DeclareMathOperator{\image}{\mathrm{Im}}

\title[Simplified numerical finite type invariant]
{Simplified numerical form of universal finite type invariant of Gauss words}

\author{Tomonori Fukunaga}
\address{Department of Mathematics, Hokkaido University, Kita 10, Nishi 8, Kita-Ku, Sapporo, Hokkaido, 060-0810, Japan}
\email{fukunaga@math.sci.hokudai.ac.jp}

\author{Takayuki Yamaguchi}
\address{Department of Mathematics, Hokkaido University, Kita 10, Nishi 8, Kita-Ku, Sapporo, Hokkaido, 060-0810, Japan}
\email{yt@math.sci.hokudai.ac.jp}

\author{Takaaki Yamanoi}
\address{Department of Mathematics, Hokkaido University, Kita 10, Nishi 8, Kita-Ku, Sapporo, Hokkaido, 060-0810, Japan}
\email{yamanoi@math.sci.hokudai.ac.jp}

\date{\today}

\begin{document}

\maketitle

\begin{abstract}
In the present paper, we study the finite type invariants of Gauss words.
In the Polyak algebra techniques,
we reduce the determination of the group structure
to transformation of a matrix into its Smith normal form and
we give the simplified form of a universal finite type invariant
by means of the isomorphism of this transformation.
The advantage of this process is that we can implement it as a computer program.
We obtain the universal finite type invariant of degree $4$, $5$, and $6$ explicitly.
Moreover, as an application,
we give the complete classification of Gauss words of rank $4$ and
the partial classification of Gauss words of rank $5$
where the distinction of only one pair remains.
\end{abstract}

\section{Introduction}
\label{sec:introduction}

One of concerns in knot theory is the classification of knots,
which is mainly studied by invariants of knots.
A finite type invariant is one of the most important classes of knot invariants.
V. A. Vassiliev introduced finite type invariants
to study topology of the space of all knots \cite{Vassiliev1992}.
Finite type invariants are also known as Vassiliev knot invariants.
Later, the definition of finite type invariants
was simplified by J. Birman and X. Lin in \cite{Birman1993}.

A knot is the image of a smooth embedding of $S^{1}$ into $\mathbb{R}^{3}$.
We also express a knot as a knot diagram in $\mathbb{R}^{2}$,
which is a smooth immersion of $S^{1}$ into $\mathbb{R}^{2}$ with
transversal double points such that the two paths at each double point
are assigned to be an over path and an under path respectively.
Such a double point is called a crossing.
Reading off labels of crossings and crossing information starting from a fixed base point,
we can interpret knots as words with additional data.
For such words, considering equivalence relations that are analogy of Reidemeister moves,
which are called homotopy moves,
we can investigate knots combinatorially.

Generalized notions of knots have been introduced.
L. Kauffman introduced the theory of virtual knots
by using combinatorially extended knot diagrams
which are called virtual knot diagrams \cite{Kauffman1999}.
A virtual knot diagram is a planar graph
of valency four endowed with the following structure: each vertex
either has an overcrossing and undercrossing
(in other words, real crossing) or is marked by a virtual crossing.
Then, he define virtual knots as the quotient of
a set of virtual knot diagrams with respect to an equivalence relation generated by
the virtual Reidemeister moves.
V. Turaev extended the theory of virtual knots and virtual links
from the aspect of Gauss codes to nanowords and nanophrases
in his papers \cite{Turaev2006,Turaev2007a}.
Nanowords are generalizations of knots and some other knot theoretical objects.
Gauss words are defined as the simplest version of nanowords and
homotopy classes of Gauss words are equivalent to free knots with a base point
in \cite{Manturov2009a,Manturov2010a}.

In \cite{Goussarov2000}, M. Goussarov, M. Polyak and O. Viro extended
the notion of the finite type invariants to virtual knots
and applied the theory of finite type invariants of classical knots to virtual knots.
They constructed an abelian group, called Polyak algebra,
and the explicit form of universal finite type invariants
that is written as count of subdiagrams for a diagram.
Furthermore, in \cite{Gibson2011}, A. Gibson and N. Ito defined
the finite type invariants of nanophrases in a similar way to Goussarov, Polyak and Viro's,
and they investigated the finite type invariants of nanophrases of lower degree.
They also obtained the finite type invariant of Gauss words of degree $4$.

Gauss words are sequences of letters so that all letters appear exactly twice or not in it.
Because Gauss words are words that have no additional information,
only configuration of letters determines homotopy classes of Gauss words.
V. Turaev conjectured that all Gauss words are homotopic to the empty word \cite{Turaev2007},
however, Gibson \cite{Gibson2011b} and Manturov \cite{Manturov2009a} proved independently
the existence of a Gauss word not homotopic to the empty word.
In \cite{Gibson2011}, the same result was also proven by means of a finite type invariant.
It is known that there exist infinite homotopy classes of such Gauss words.
Because nanowords are Gauss words with additional data,
we regard Gauss words as invariants of knots and of their generalizations.

An argument of Polyak algebra is useful for studies using computer programs.
For virtual knots, computing the group defined from Polyak algebra,
D. Bar-Natan et al.\ conjectured relations between dimensions of spaces of finite type invariants
and dimensions of spaces of weight systems \cite{Bar-Natan2011}.
Note that their computation was carried out when Polyak algebra is defined over some field
while Polyak algebra in our situation is defined over $\mathbb{Z}$.

The universal finite type invariant of Gauss words
can be written as a map counting isomorphic subwords of a word.
In this paper, we present a simplified form of the universal finite type invariant,
which is a map between words and a direct sum of cyclic groups.
To obtain the form of the invariant,
we determine the structure of the abelian group $H_n$,
which is defined by truncating words of rank more than $n$ in Polyak algebra.
We can reduce this determination of the group structure
to transformation of a matrix into its Smith normal form,
which is a popular technique in computer algebra.
Then, we obtain a direct sum of cyclic groups that is isomorphic to $H_n$.
Also, the transformation matrix in this process
gives the isomorphism between the group $H_n$ and the direct sum of cyclic groups.
Composing the isomorphism with the original universal finite type invariant,
we have the simplified form.

Our computer program to carry out this process determined
the group structure for rank less than or equal to $7$.
In addition, we constructed the universal invariants of degree $4$, $5$, and $6$.
As an application of these invariants,
we show the partial classification of Gauss words of rank less than or equal to $5$.

This paper is organized as follows. In Section \ref{sec:definitions-notations},
we review the definitions and notations on Gauss words and finite type invariants.
We recall Ito and Gibson's results in \cite{Gibson2011}.
In Section \ref{sec:determination-group},
we present an algorithm to determine the group structure
and give our computational results for truncated Polyak algebra.
In Section \ref{sec:simplified-form}, we transform the universal finite type invariants
into the simplified numerical form and
present explicitly the finite type invariant of the degree $4$ and $5$.
In section \ref{sec:homotopy-classification},
we apply our main result to the classification of Gauss words of at most $5$ letters.

\section{Gauss words and finite type invariants}
\label{sec:definitions-notations}

In this section, we recall some definitions and facts on Gauss words
and finite type invariants of them.
The definitions and notations in this section are those of
a restricted version of \cite{Gibson2011}.
We can also refer to \cite{Turaev2006,Turaev2007,Turaev2007a}
for basic definitions.

A \emph{word} of length $n$ is a sequence of $n$ letters.
If a word is length $0$ we call it an \emph{empty word}.
In this paper we consider \emph{Gauss words},
which are sequences of letters so that
all letters appear exactly twice or not in it.
The \emph{rank} of a Gauss word is the number of distinct letters appearing in it.
Clearly, the rank of a Gauss word is the half of length of it.
Two Gauss words $w_1$ and $w_2$ are \emph{isomorphic}
if there is a bijection between the sets of letters appearing in $w_1$ and $w_2$
so that a word created by mapping all letters of $w_1$ coincides with $w_2$.

We define \emph{homotopy moves} for Gauss words as the following.
\begin{itemize}
  \item[H1:] $xAAy \leftrightarrow xy$
  \item[H2:] $xAByBAz \leftrightarrow xyz$
  \item[H3:] $xAByACzBCt \leftrightarrow xBAyCAzCBt$
\end{itemize}
Here, $x$, $y$, $z$, and $t$ are arbitrary words that can be empty words.
\emph{Homotopy} is the equivalence relation generated by isomorphisms
and homotopy moves of the three types.

Let $P$ be the set of homotopy classes of Gauss words and
$\mathbb{Z} P$ be the free abelian group generated by the elements of $P$.
We define a \emph{semi-letter} $\dot{A}$ for a Gauss word of the form $xAyAz$,
where $x$, $y$, and $z$ are arbitrary words that can be empty words.
A word including $\dot{A}$ in $\mathbb{Z} P$ is defined by
\begin{align}
  x\dot{A}y\dot{A}z = xAyAz - xyz.
\end{align}
Semi-letters define a class of homotopy invariants as follows.
Let $v$ be a homotopy invariant for Gauss words
taking values in an abelian group.
We extend linearly a homotopy invariant to $\mathbb{Z} P$.
A homotopy invariant $v$ is a \emph{finite type invariant}
if there exists an integer $n$ so that $v(w) = 0$ for any Gauss word $w$
including more than $n$ semi-letters.
We call such a least integer $n$ the \emph{degree} of the finite type invariant $v$.
A finite type invariant $v: \mathbb{Z} P \to G$ of degree $n$ is
a \emph{universal invariant} of degree $n$
if for any finite type invariant $v^\prime$ of degree $n: \mathbb{Z} P \to H$,
there exists a homomorphism $f$ so that the following diagram is commutative.

\begin{align}
\begin{minipage}[c]{3cm}
\xymatrix{
\mathbb{Z} P \ar[r]^-v \ar[rd]_{v^\prime} &
G \ar[d]^{f} \\
&H
}
\end{minipage}
\end{align}

Goussarov et al.\ described a universal invariant of virtual knots
as an angle bracket formula \cite{Goussarov2000}.
Gibson and Ito \cite{Gibson2011} extended the formula to nanophrases.
Let $w_1$ and $w_2$ be words.
The word $w_1$ is a \emph{subword} of the word $w_2$
if we can generate $w_1$ by deleting some letters from $w_2$.
Then we write $w_1 \triangleleft w_2$.
For any word $w$, $w$ itself and the empty word are subwords of $w$.
We define an angle bracket $\langle w_1, w_2 \rangle$
to be the number of subwords of $w_2$ isomorphic to $w_1$.
We extend linearly the angle bracket to $\mathbb{Z} P$.

Let $\mathbb{Z} \mathcal{I}$ be the free abelian group generated
by $\mathbb{Z}$ isomorphism classes of Gauss words.
Then $\mathbb{Z} P$ coincides with $\mathbb{Z} \mathcal{I}$ modulo homotopy moves H1 to H3.
We consider three other relations on $\mathbb{Z} \mathcal{I}$ than the homotopy moves.
\begin{itemize}
  \item[G1:] $xAAy = 0$
  \item[G2:] $xAByBAz + 2xAyAz = 0$
  \item[G3:] $xAByACzBCt + xAByAzBt + xAyACzCt + xByCzBCt = xBAyCAzCBt + xBAyAzBt + xAyCAzCt + xByCzCBt$
\end{itemize}
Here, $x$, $y$, $z$, and $t$ are arbitrary words
so that each term of word in the above relations is a Gauss word.
Let $G$ be the group given by $\mathbb{Z} \mathcal{I}$ modulo these three types of relations,
which is called Polyak algebra.
Note that Polyak algebra has the structure of an algebra defined by concatenation of words,
however, we do not use it in this paper.
To define a group $G_n$,
we introduce another relation G4 for a positive integer $n$:
\begin{itemize}
  \item[G4:] If the rank of a word $w$ is greater than $n$ then $w = 0$.
\end{itemize}
Note that G4 depends on the rank $n$.
Let $G_n$ be the group given by $\mathbb{Z} \mathcal{I}$ modulo the relations G1, G2, G3, and G4.
We can decompose $G_n$ into the direct sum of two groups:
one is the group generated by the empty word
and the other is the group generated by the other generators.
Because the empty word does not appear in the relations G1 to G4, it generates $\mathbb{Z}$.
Therefore, we have
\begin{align}
 G_n = \mathbb{Z} \oplus H_n, \label{eq:gn-decomposition}
\end{align}
where $H_n$ be the group generated by Gauss words except
the empty word whose rank is less than or equal to $n$.

We define a map $\theta: \mathbb{Z} \mathcal{I} \to \mathbb{Z} \mathcal{I}$ by
\begin{align}
  \theta (p) = \sum_{q \triangleleft p} q.
\end{align}
We also define an additive map $O_n: \mathbb{Z} \mathcal{I} \to \mathbb{Z} \mathcal{I}$
by $O_n(p) = p$ for a word $p$ of rank less than or equal to $n$
and $O_n (p) = 0$ for a word $p$ of rank greater than $n$.
Then, the map $\theta$ induces an isomorphism $\hat{\theta}$ from $\mathbb{Z} P$ onto $G$
(see Proposition 5.7 in \cite{Gibson2011}).
Also, the map $O_n$ induces a map from $G$ to $G_n$.
We define the map $\tilde{\Gamma}_n: \mathbb{Z} P \to G_n$
as the composite of $\hat{\theta}$ and $O_n$:
\begin{align}
  \begin{minipage}[c]{5cm}
    \xymatrix{
    \mathbb{Z} \mathcal{I} \ar[r]^{\theta} \ar[d] & \mathbb{Z} \mathcal{I} \ar[r]^{O_n} \ar[d] &
    \mathbb{Z} \mathcal{I} \ar[d] \\
    \mathbb{Z} P \ar[r]^{\hat{\theta}} \ar@(dr,dl)[rr]_{\tilde{\Gamma}_n} & G \ar[r]^{O_n} & G_n .
    }
  \end{minipage}
  \label{diagram_of_situation}
\end{align}
Let $P_n$ be the set of homotopy classes of Gauss words whose rank is $n$ or less.
We also have
\begin{align}
  \tilde{\Gamma}_n (p) = \sum_{q \triangleleft p} O_n(p) = \sum_{q \in P_n} \langle q, p \rangle q.
\end{align}
The following result is crucial.
\begin{proposition}[Gibson and Ito \protect{\cite[Proposition 5.9]{Gibson2011}}]
  The map $\tilde{\Gamma}_n$ is a universal invariant of degree $n$.
\end{proposition}
We define $\Gamma_n: \mathbb{Z} P \to H_n$ as the composite of $\tilde{\Gamma}_n$ and
the natural map from $G_n$ to $H_n$.
Because for $\sum a_i w_i$ in $\mathbb{Z} P$ where $a_i$ are integers
\begin{align}
  \tilde{\Gamma}_n \left( \sum a_i w_i \right) = \left( \sum a_i, \Gamma_n \left( \sum a_i w_i \right) \right),
\end{align}
to analyze $\tilde{\Gamma}_n$ is equivalent to analyzing $\Gamma_n$.
Later we give a numerical form of $\Gamma_n$.

The structures of $G_n$ for smaller $n$ have been determined.
\begin{proposition}[Gibson and Ito \protect{\cite{Gibson2011}}]
  For Gauss words, $G_0$, $G_1$, $G_2$, and $G_3$ are isomorphic to $\mathbb{Z}$ and
  $G_4$ is isomorphic to $\mathbb{Z} \oplus \mathbb{Z}/2\mathbb{Z}$.
\end{proposition}
We determine the structures of $G_n$ for $n = 5, 6, 7$ in the next section.

\section{Computational determination of the group $H_n$}
\label{sec:determination-group}

It is well-known that every finitely generated abelian group is (noncanonically) isomorphic to
the direct sum of cyclic groups, and when we are given generators and relations on the group,
we can construct an isomorphism in terms of matrix theory
(See \cite{Lang2002} for example).
We rephrase it in a suitable form for our situation.

\begin{proposition}
  Let $M$ be a free abelian group with basis $w_1, \dots, w_s$.
  Let $N$ be the subgroup generated by $\{ \sum_{i=1}^s a_{ji} w_i \mid j = 1, \dots, s \}$ for $a_{ji} \in \mathbb{Z}$ and $j = 1, \dots, t$,
  which corresponds to relations of $M$ so that $\sum_{i=1}^s a_{ji} w_i = 0$.
  Then, there are a nonzero integer $p$ and
  positive integers $d_1, \dots, d_l$ with $l = s - p$ such that $M/N$ is isomorphic to
  \begin{align}
    \mathbb{Z}^{p} \oplus \bigoplus^{l}_{i = 1} (\mathbb{Z} / d_i\mathbb{Z}), \label{eq:module_decomposition}
  \end{align}
  and such that if $l$ is greater than 1 then $d_{i-1}$ divides $d_i$ for $i = 2, \dots, l$.
  These integers $d_1, \dots, d_l$ and $p$ are uniquely determined.

  Moreover, if $A$ denotes the $s \times t$ matrix $(a_{ij})$ and $S$ denotes the $s \times t$ matrix
  \begin{align}
    \left(
    \begin{array}{ccccc|c}
      d_1 & & & & & \\
      & d_2 & & 0 & & 0 \\
      0 & & \ddots & & & \\
      & &  & d_l & & \\ \hline
	&  & 0 & & & 0 \\
    \end{array}
    \right)
    \label{Smith_normal_form}
  \end{align}
  then we can transform $A$ into $S$ by row and column operations,
  where row (respectively column) operations mean either
  interchanging two rows (respectively columns),
  adding a multiple of a row (respectively column) to another,
  or multiplying a row (respectively column) by $-1$.
  In other words, there are a $\mathbb{Z}$-invertible $s \times s$ matrix $U$
  and a $\mathbb{Z}$-invertible $t \times t$ matrix $V$
  so that $S = UAV$.
  The matrix $S$ is called the Smith normal form of $A$ and
  $d_i$ are called its elementary divisors.
\end{proposition}

Because the group $H_n$ is finitely generated,
it is isomorphic to a group of the form (\ref{eq:module_decomposition}).
The second type relations G2 restrict possibilities of the structure of $H_n$ further.

\begin{proposition}
  \label{prop_form_Hn}
  The group $H_n$ is isomorphic to a group of the form
  \begin{align}
    \bigoplus^{n - 3}_{i = 1} (\mathbb{Z} / 2^{i}\mathbb{Z})^{p_i}, \label{eq:grp_form}
  \end{align}
  for some nonnegative integers $p_i$.

  In particular, for an arbitrary Gauss word $w$ of rank $m$ in $H_n$, we have
  \begin{align}
    2^{n - m + 1} w = 0. \label{eq:two_power}
  \end{align}
\end{proposition}

\begin{proof}
  From second type relation G2, for an arbitrary Gauss word $w$ of rank $m$,
  there exists a Gauss word $w'$ of rank $m + 1$ so that $w' + 2w = 0$.
  If $m = n$ then the word $w'$ has rank $n + 1$ and $w' = 0$ in $H_n$.
  Therefore, we have $2w = 0$.
  Easy induction proves the second assertion for $m < n$.
  Obviously, under the relation (\ref{eq:two_power}), $H_n$ must be of the form as in (\ref{eq:grp_form}).
\end{proof}

To determine the structure of $H_n$
is equivalent to computing the Smith normal form of the matrix obtained from
the generators and the relations.
To be more precise, we carried out the following procedures and determined $H_n$.
for $n = 5, \dots, 7$

\begin{enumerate}
  \item Generate all Gauss words whose rank is less than or equal to $n+1$ excluding the empty word.
  \item Remove words of the form $xAAy$ and words of rank $n+1$ from the set of Gauss words generated at the first step.
  \item Keep all Gauss words generated at the second step as a set of generators for $H_n$.
  \item Generate all relations for the set of Gauss words generated at the first step;
	to be more precise,
	generate second type relations from Gauss words of the form $xAByBAz$ and
	generate third type relations from Gauss words of the form $xAByACzBCt$.
  \item Remove words of the form $xAAy$ and words of rank $n+1$ from the set of relations generated at the fourth step.
  \item Define a matrix from generators obtained at the third step and relations obtained at the fifth step.
  \item Compute the Smith normal form of the matrix by using row and column operations.
\end{enumerate}

Smith normal form computation has two main difficulties.
One is coefficient growth: the absolute value of a coefficient becomes larger and larger
as the transformation proceeds.
Such a large coefficient raises an overflow error on computation.
Fortunately, Proposition \ref{prop_form_Hn} allows us
to carry out the transformation on $\mathbb{Z}/(2^{n-1})\mathbb{Z}$.

\begin{table}[htb]
  \begin{tabular}{|r|r|r|r|r|}
    \hline
    & \shortstack{number of \\ generators} & \shortstack{number of second \\ type relations} &
	\shortstack{number of third \\ type relations} & \shortstack{number of total \\ unique relations} \\
    \hline
    $H_4$ & 42     & 161     & 62      & 97 \\
    $H_5$ & 371    & 1806    & 672     & 998 \\
    $H_6$ & 4026   & 23736   & 8652    & 12287 \\
    $H_7$ & 51870  & 358644  & 128926  & 176591 \\
    $H_8$ & 773185 & 6129164 & 2181235 & 2900594 \\
    \hline
  \end{tabular}
  \caption{Numbers of generators and relations for $H_n$.
  To obtain a group isomorphic to $H_n$, we need to transform a matrix into Smith normal form.
  The row size of the matrix is the number of generators and the column size is the number of total unique relations.
  Note that numbers of relations in the second and third columns count duplicated relations.}
  \label{number_generators_relations}
\end{table}

The other difficulty is ``fill-in'' on transformation of a sparse matrix,
which occurs also on Gaussian elimination.
The matrix obtained from our relations and generators is extremely sparse
because the numbers of relations and generators are very large (Table \ref{number_generators_relations})
while the relations have at most only 8 terms.
The sparsity of a matrix is getting lost gradually during the operations of the matrix.
This is quite a difficult problem and has been studied
as an elimination game of a chordal graph on graph theory \cite{Heggernes2006}.
In our case, because about 50 \% of relations have only one or two terms,
we can delay fill-in by eliminating these relations at the beginning.

We consider the matrix size as a rough estimate of computation amount.
The number of isomorphism classes of rank $n$ Gauss words
is $(2n - 1) \cdot (2n - 3) \cdots 3 \cdot 1$.
The number of generators excluding words of the form $xAAy$
is also multiplied by about $2n + 1$
as the rank $n$ increases to $n + 1$.
The growth of the number of corresponding relations is
of the same order as the number of generators.
Then, the number of the nonzero matrix entries increases roughly at an order of $4n^2$.
Table \ref{number_generators_relations} presents the actual number of generators and
the number of relations in our computation.

The growth of computation for the rank prevented us from determining the group $H_8$;
in fact, our computer program to obtain $H_7$ took more than one week and
it probably takes hundreds of days to obtain $H_8$.
To conclude this section,
we give our computational result for $H_5$, $H_6$, and $H_7$ as a proposition.
\begin{proposition}
  \label{prop_numerical_result}
  We have the following.
  \begin{align}
    G_5 & \cong \mathbb{Z} \oplus (\mathbb{Z}/2\mathbb{Z})^6 \oplus \mathbb{Z}/4\mathbb{Z} \\
    G_6 & \cong \mathbb{Z} \oplus (\mathbb{Z}/2\mathbb{Z})^{32} \oplus (\mathbb{Z}/4\mathbb{Z})^6 \oplus \mathbb{Z}/8\mathbb{Z} \\
    G_7 & \cong \mathbb{Z} \oplus (\mathbb{Z}/2\mathbb{Z})^{188} \oplus (\mathbb{Z}/4\mathbb{Z})^{32}
    \oplus (\mathbb{Z}/8\mathbb{Z})^6 \oplus \mathbb{Z}/16\mathbb{Z}
  \end{align}
\end{proposition}

\section{Simplified form of the universal invariant}
\label{sec:simplified-form}

Let $Q_n = \{ w_1, \dots, w_s \}$ be the finite set of isomorphism classes
of Gauss words of rank $n$ or less excluding the empty word
and let an integer $t$ be the number of relations.
Then, $\{ w_1, \dots, w_s \}$ is the basis of $\mathbb{Z} Q_n$ and
$\mathbb{Z} Q_n$ is isomorphic to $\mathbb{Z}^s$.
Defining a matrix $A$ from the generators of $\mathbb{Z} Q_n$ and the relations,
we transformed the matrix $A$ into the Smith normal form $S$
by using row operations $U$ and column operations $V$ in the last section.
The matrix $U$ induces an isomorphism between $\mathbb{Z}^s/\image A$ and $\mathbb{Z}^s/\image S$.
From the definitions of $H_n$ and $A$, it holds that $H_n$ is isomorphic to $\mathbb{Z}^s/\image A$.

\begin{align}
  \begin{minipage}[c]{5cm}
    \xymatrix{
      & \mathbb{Z} Q_n \ar[d]^{\cong} \ar[r] & H_n \ar[d]^{\cong} \\
    \mathbb{Z}^t \ar[r]^{A}  & \mathbb{Z}^s \ar[d]^{U} \ar[r] & \mathbb{Z}^s/\image A \ar[d]^{U} \\
    \mathbb{Z}^t \ar[u]_{V}  \ar[r]_{S} & \mathbb{Z}^s \ar[r] & \mathbb{Z}^s/\image S
    }
  \end{minipage}
  \label{diagram_smith}
\end{align}

Because the map $\Gamma_n: \mathbb{Z} P \to H_n$ is a universal invariant of degree $n$,
so is the composite of $\Gamma_n$ and the isomorphism $H_n \cong \mathbb{Z}^s/\image S$.
To exhibit this universal invariant, we need to determine the image of each element of $Q_n$.
Because the map
\begin{align}
  U \circ \Gamma_n (p) = U \left( \sum_{i = 1}^s \langle w_i, p \rangle w_i \right) = \sum_{i = 1}^s \langle w_i, p \rangle U(w_i)
\end{align}
gives the universal invariant from $\mathbb{Z} P$ to $\mathbb{Z}^s/\image S$,
it is sufficient to calculate $U(w_i)$ in order to obtain the explicit form.

Suppose that the Smith normal form $S = UAV$ has the form (\ref{Smith_normal_form}),
i.e., the elementary divisors of the matrix $A$ are $d_1, \dots, d_l$.
It follows from Proposition \ref{prop_form_Hn} that we have $l = s$ and
$d_j = 2^{q_j}$ for some nonnegative integer $q_j$ and $j = 1, \dots, s$.
If $d_j = 1$ for all $j$, then $U(w_i) = 0$ in $\mathbb{Z}^N/\image S$ for all $i$
and hence the invariant $U \circ \Gamma_n$ is trivial.
This situation occurs in the case of $\Gamma_n$ for $n = 0, \dots, 3$.
We consider the case when $d_{k-1} = 1$ and $d_k \ne 1$.
Under the identification $\mathbb{Z} Q_n \cong \mathbb{Z}^s$,
$w_i$ is assumed to map to the $i$-th unit vector $e_i = (0, \dots, 0, 1, 0, \dots, 0)^T$,
where $v^T$ denotes the transpose of a vector $v$.
Let $v_i$ be the $(s-k+1)$-dimensional vector consisting of the last $s-k+1$ entries of $Ue_i$
i.e., $v_i = (x_k, \dots, x_s)^T$ if we write $Ue_i = (x_1, \dots, x_s)^T$.
From the diagram (\ref{diagram_smith}), we see that
$U(w_i) \ne 0$ in $\mathbb{Z}^s/\image S$
if and only if $v_i \ne 0$ in $\oplus_{j=k}^s \mathbb{Z}/d_j\mathbb{Z}$.
Therefore, by adding subscript $n$ to $U$, $s$, $w_i$, and $v_i$ in order to point out dependency on
the degree $n$ of the finite type invariant,
we obtain the simplified numerical form $\bar{\Gamma}_n$ of the universal invariant $\Gamma_n$:
\begin{align}
  \bar{\Gamma}_n(p) = U_n \circ \Gamma_n (p) = \sum_{i = 1}^{s_n} \langle w_{n, i}, p \rangle v_{n, i}.
\end{align}
Note that the map $\bar{\Gamma}_n$ is a map from $\mathbb{Z} P$ to $\oplus_{j=k}^s \mathbb{Z}/d_j\mathbb{Z}$.

We obtained the universal finite type invariant for $n = 4, 5$ from our computation in this way.
All words of rank $4$ or less except the following words $w_1, \dots, w_6$
map to $0$ in $\mathbb{Z}^{s_4}/\image S$ via $U_4$:
\begin{align}
  \begin{array}{ll}
    w_1 = ABACDCBD, & w_2 = ABCACDBD, \\
    w_3 = ABCADBDC, & w_4 = ABCBDACD, \\
    w_5 = ABCDBDAC, & w_6 = ABCDCADB.
  \end{array}
  \label{g4_non_zero_words}
\end{align}
Then, we obtain the explicit form of $\bar{\Gamma}_4$,
which is the same as the invariant of Proposition 8.2 in \cite{Gibson2011b}.

\begin{proposition}[Gibson and Ito \cite{Gibson2011b}]
  We define $w_1, \dots, w_6$ by (\ref{g4_non_zero_words}).
  The map $\bar{\Gamma}_4: \mathbb{Z} P \to \mathbb{Z}/2\mathbb{Z}$ defined by
  \begin{align}
    \bar{\Gamma}_4(p) = \left( \sum_{i = 1}^{6} \langle w_i, p \rangle \right) \bmod 2
  \end{align}
  is the universal finite type invariant of degree $4$.
\end{proposition}

\begin{table}
  \scalebox{0.75}{
  \begin{tabular}{|c|lll|}
\hline
Vector in $(\mathbb{Z}/2\mathbb{Z})^6 \oplus \mathbb{Z}/4\mathbb{Z}$ &
\multicolumn{3}{c|}{Words of rank $5$ or less that are not $0$ in $G_5$} \\
\hline
$(0, 0, 0, 0, 0, 0, 1)$ &
ABACDCBD & ABCBDACD & ABCDCADB \\ \hline
$(0, 0, 0, 0, 0, 0, 2)$ &
ABACDECBDE & ABACDECBED & ABACDECDBE \\
& ABACDEDCBE & ABCABDEDCE & ABCACDEBDE \\
& ABCACDEBED & ABCADEBDEC & ABCADEBEDC \\
& ABCADEDBCE & ABCADEDCBE & ABCBADEDCE \\
& ABCBDEACDE & ABCBDEACED & ABCBDEADCE \\
& ABCBDEDACE & ABCDABDECE & ABCDABECED \\
& ABCDACDEBE & ABCDACEBED & ABCDADCEBE \\
& ABCDADEBCE & ABCDADEBEC & ABCDADECBE \\
& ABCDAEBCED & ABCDAEBECD & ABCDAEBEDC \\
& ABCDAECBED & ABCDAECEBD & ABCDBADECE \\
& ABCDBAECED & ABCDBCEADE & ABCDBEACDE \\
& ABCDBEADCE & ABCDBECEAD & ABCDCBEADE \\
& ABCDCEABDE & ABCDCEADBE & ABCDCEADEB \\
& ABCDCEBADE & ABCDEBCEAD & ABCDEBDEAC \\
& ABCDEBEACD & ABCDEBEADC & ABCDEBEDAC \\
& ABCDECADEB & ABCDECAEDB & ABCDECBEAD \\
& ABCDECDAEB & ABCDECEABD & ABCDECEBAD \\
& ABCDEDABEC & ABCDEDAEBC & ABCDEDAECB \\
& ABCDEDBAEC & ABCDEDCAEB & \\ \hline
$(0, 0, 0, 0, 0, 1, 2)$ &
ABCDEDBEAC & & \\ \hline
$(0, 0, 0, 0, 1, 0, 2)$ &
ABCDECEADB & & \\ \hline
$(0, 0, 0, 1, 0, 0, 0)$ &
ABCACDBEDE & ABCBDEAECD & ABCDCEAEDB \\
& ABCDEDACEB & & \\ \hline
$(0, 0, 0, 1, 0, 1, 2)$ &
ABCADCDEBE & & \\ \hline
$(0, 0, 0, 1, 1, 0, 0)$ &
ABCADECDBE & ABCADEDBEC & ABCDBECADE \\ \hline
$(0, 0, 0, 1, 1, 0, 2)$ &
ABCADBEDEC & ABCDCEDAEB & \\ \hline
$(0, 0, 0, 1, 1, 1, 0)$ &
ABACDBDECE & ABACDBECED & ABACDCEBED \\
& ABACDEDBCE & ABCBDEDCAE & ABCDCAEDEB \\
& ABCDCEBDAE & & \\ \hline
$(0, 0, 1, 0, 0, 0, 0)$ &
ABCACDEDBE & ABCADBECDE & ABCDBEACED \\
& ABCDECAEBD & & \\ \hline
$(0, 0, 1, 0, 1, 1, 0)$ &
ABCBDAECED & & \\ \hline
$(0, 0, 1, 1, 0, 0, 0)$ &
ABACDCBEDE & ABCBDACEDE & ABCBDADECE \\
& ABCBDECEAD & ABCDADBECE & ABCDADECEB \\
& ABCDAECEDB & & \\ \hline
$(0, 0, 1, 1, 0, 1, 2)$ &
ABACDCEDBE & & \\ \hline
$(0, 0, 1, 1, 1, 0, 0)$ &
ABCDACEDBE & ABCDAEBDEC & ABCDBDEACE \\ \hline
$(0, 0, 1, 1, 1, 0, 2)$ &
ABCBDCEADE & ABCDBEDEAC & \\ \hline
$(0, 0, 1, 1, 1, 1, 0)$ &
ABACBDEDCE & ABCDCAEBED & ABCDCEBEAD \\
& ABCDEBECAD & & \\ \hline
$(0, 1, 0, 0, 0, 0, 0)$ &
ABCDBECAED & ABCDEACEBD & ABCDEADBEC \\ \hline
$(0, 1, 0, 0, 0, 0, 2)$ &
ABCBDEDAEC & & \\ \hline
$(0, 1, 0, 0, 1, 0, 0)$ &
ABCADEDCEB & & \\ \hline
$(0, 1, 0, 1, 0, 1, 2)$ &
ABACDCEBDE & ABACDEDBEC & ABCADBDECE \\ 
& ABCDBDAECE & & \\ \hline
$(0, 1, 0, 1, 1, 0, 0)$ &
ABACDBCEDE & ABACDBEDEC & ABCBDCAEDE \\ \hline
$(0, 1, 1, 0, 0, 0, 0)$ &
ABCDBEADEC & ABCDEBDACE & ABCDECADBE \\ \hline
$(0, 1, 1, 0, 0, 0, 2)$ &
ABCDBDEAEC & & \\ \hline
$(0, 1, 1, 0, 0, 1, 0)$ &
ABCDBDECAE & & \\ \hline
$(0, 1, 1, 1, 0, 0, 2)$ &
ABCBDCEAED & ABCDCADEBE & \\ \hline
$(0, 1, 1, 1, 1, 0, 0)$ &
ABACBDCEDE & ABCADCEDEB & ABCADEBDCE \\
& ABCBDCEDAE & ABCDACEBDE & ABCDAEDBEC \\
& ABCDBEDACE & ABCDBEDAEC & ABCDEBDAEC \\
& ABACBDCD & & \\ \hline
$(0, 1, 1, 1, 1, 0, 3)$ &
ABCACDBD & ABCADBDC & ABCDBDAC \\ \hline
$(0, 1, 1, 1, 1, 1, 2)$ &
ABACDECEBD & ABCBDAEDEC & \\ \hline
$(1, 0, 0, 0, 0, 0, 0)$ &
ABCADBECED & ABCADECEBD & ABCBDAECDE \\
& ABCBDEDCEA & ABCDBDECEA & ABCDBECEDA \\
& ABCDCAEDBE & ABCDCEBDEA & ABCDECEBDA \\
& ABCDEDBECA & & \\ \hline
$(1, 0, 0, 0, 1, 1, 2)$ &
ABCADCEBED & ABCADEBECD & ABCBDAEDCE \\
& ABCDCAEBDE & & \\ \hline
  \end{tabular}
  }
  \caption{Words of rank $5$ mapping to nonzero elements in
  $(\mathbb{Z}/2\mathbb{Z})^6 \oplus \mathbb{Z}/4\mathbb{Z}$ via the isomorphism $U_5$
  and their values. In Proposition \ref{prop:gamma5} we let $W_5$ be the set of words in the right column and
  $v(w)$ be the vector in the left column corresponding to $w$ in $W_5$.}
  \label{g5_non_zero_words}
\end{table}

For $n = 5$ Table \ref{g5_non_zero_words}
shows pairs of words and nonzero vectors in
$(\mathbb{Z}/2\mathbb{Z})^6 \oplus \mathbb{Z}/4\mathbb{Z}$.
This is enough to obtain $\bar{\Gamma}_5$.

\begin{proposition}
  \label{prop:gamma5}
  Let $W_5$ be the set of words appearing in Table \ref{g5_non_zero_words} and
  $v(w)$ be the corresponding vector in $(\mathbb{Z}/2\mathbb{Z})^6 \oplus \mathbb{Z}/4\mathbb{Z}$
  for $w$ in $W_5$.
  The map $\bar{\Gamma}_5: \mathbb{Z} P \to (\mathbb{Z}/2\mathbb{Z})^6 \oplus \mathbb{Z}/4\mathbb{Z}$ defined by
  \begin{align}
    \bar{\Gamma}_5(p) = \sum_{w \in W_5} \langle w_i, p \rangle v(w)
  \end{align}
  is the universal finite type invariant of degree $5$.
\end{proposition}

\begin{remark}
  The expression of $\Gamma_5$ is not unique.
  Indeed, the values $v(w)$ depends on the choice of an isomorphism
  $H_5 \cong (\mathbb{Z}/2\mathbb{Z})^6 \oplus \mathbb{Z}/4\mathbb{Z}$.
\end{remark}

\section{Homotopy classification of Gauss words of rank $4$ and $5$}
\label{sec:homotopy-classification}

As an application of the simplified form of the universal invariant,
we classify completely Gauss words of rank $4$.
We also apply the universal invariant of degree $6$ to Gauss words of rank $5$
and we obtain the classification with only one unclassified pair of two Gauss words.
We recall moves derived from H1, H2, and H3.

\begin{lemma}[Turaev \cite{Turaev2007}]
  The following pairs of words are in same homotopy class of words:
  \begin{itemize}
    \item[H4:] $xAByABz \leftrightarrow xyz$,
    \item[H5:] $xAByCAzBCt \leftrightarrow xBAyACzCBt$,
    \item[H6:] $xAByCAzCBt \leftrightarrow xBAyACzBCt$,
    \item[H7:] $xAByACzCBt \leftrightarrow xBAyCAzBCt$,
  \end{itemize}
  where $x$, $y$, $z$, and $t$ are arbitrary words that can be empty words.
\end{lemma}

Here, we give the classification of Gauss words of rank $4$ or less
by using homotopy moves H1 to H7 and $\bar{\Gamma}_5$.

\begin{theorem}
  Gauss words of rank $4$ or less are classified into the following four homotopy classes:
  \begin{itemize}
    \item $\{ ABACDCBD, ABCBDACD, ABCDCADB \}$,
    \item $\{ ABCACDBD, ABCADBDC, ABCDBDAC \}$,
    \item $\{ ABACBDCD \}$, and
    \item the set of other words.
  \end{itemize}
\end{theorem}

\begin{proof}
  By using the moves from H1 to H7,
  we can easily show that words of rank at most $4$
  except for the seven words exhibited in the list of the statement
  are homotopic to the empty word.
  For the first two sets in the list,
  the elements in each set are homotopic to each other via the moves H3, H5, H6, or H7.
  Applying the universal finite type invariant $\bar{\Gamma}_5$, we have
  \begin{align}
    \bar{\Gamma}_5(ABACDCBD) & = (0, 0, 0, 0, 0, 0, 1), \\
    \bar{\Gamma}_5(ABCACDBD) & = (0, 1, 1, 1, 1, 0, 3), \\
    \bar{\Gamma}_5(ABACBDCD) & = (0, 1, 1, 1, 1, 0, 0), \\
    \bar{\Gamma}_5(\emptyset) & = (0, 0, 0, 0, 0, 0, 0).
  \end{align}
  Therefore, we see that these four sets are distinct homotopy classes.
\end{proof}

\begin{table}
  \begin{tabular}{|llll|}
\hline
\multicolumn{4}{|l|}{$\{ \text{ABCDBEACED}, \text{ABCDECAEBD} \}$ and $\{ \text{ABCADBECDE}, \text{ABCACDEDBE} \}$} \\
\hline
ABCBDAEDCE & ABCDCAEBDE & & \\
\hline
ABCACDBEDE & & & \\
\hline
ABACDCEDBE & & & \\
\hline
ABACDCEBED & ABCDCAEDEB & ABCBDEDCAE & ABACDEDBCE \\
ABCDCEBDAE & & & \\
\hline
ABCBDCEAED & & & \\
\hline
ABCDECEADB & & & \\
\hline
ABCDCEDAEB & ABCDEBDACE & & \\
\hline
ABCDEBECAD & ABCDCEBEAD & ABCDCAEBED & \\
\hline
ABCADBDECE & ABCDBDAECE & & \\
\hline
ABACDCEBDE & ABACDEDBEC & & \\
\hline
ABCADCDEBE & & & \\
\hline
ABACBDCEDE & & & \\
\hline
ABACDBECED & ABACDBDECE & & \\
\hline
ABCDBEDEAC & ABCDEADBEC & & \\
\hline
ABCBDCEDEA & & & \\
\hline
ABACDBCEDE & ABCBDCAEDE & ABACDBEDEC & \\
\hline
ABCADCEBED & ABCADEBECD & & \\
\hline
ABCADEDCEB & & & \\
\hline
ABCBDCEADE & ABCBDEDAEC & ABCDACEDBE & ABCDAEBDEC \\
\hline
ABCDEDBECA & ABCBDAECDE & ABCBDEDCEA & ABCDCAEDBE \\
ABCDCEBDEA & & & \\
\hline
ABCDBECAED & ABCDBDEACE & ABCDEACEBD & \\
\hline
ABCDBEADEC & ABCDECADBE & ABCADEDBEC & \\
\hline
ABCDCADEBE & & & \\
\hline
ABCADECDBE & ABCADBEDEC & ABCDBECADE & ABCDBDEAEC \\
\hline
ABCADECEBD & ABCDBECEDA & ABCDECEBDA & ABCDBDECEA \\
ABCADBECED & & & \\
\hline
ABACDECEBD & & & \\
\hline
ABCDBDECAE & & & \\
\hline
ABCBDEAECD & ABCDEDACEB & ABCDCEAEDB & \\
\hline
ABACBDEDCE & & & \\
\hline
ABCDADECEB & ABCDAECEDB & ABCBDADECE & ABCDADBECE \\
ABCBDECEAD & & & \\
\hline
ABCBDAECED & & & \\
\hline
ABCBDAEDEC & & & \\
\hline
ABCDEDBEAC & & & \\
\hline
ABACDCBEDE & ABCBDACEDE & & \\
\hline
ABACDCBD & ABCBDACD & ABCDCADB & \\
\hline
ABCACDBD & ABCADBDC & ABCDBDAC & \\
\hline
ABACBDCD & & & \\
\hline
  \end{tabular}
  \caption{Partial classification of Gauss words of rank $5$ or less under the universal invariant $\bar{\Gamma}_6$.
  Words excluded from the table are homotopic to the empty word or words of rank at most $4$.
  The invariant $\bar{\Gamma}_6$ distinguishes words in a row from words in another row.
  The words in each row except the first row are homotopic to each other.
  The classification of the two sets in the first row remains unknown.}
  \label{g6_classified_non_zero_words}
\end{table}

The universal invariant $\bar{\Gamma}_6$ obtained from our computation
gives the following theorem, which shows
a partial classification of Gauss words of rank $5$ or less.
We omit the specific description of $\bar{\Gamma}_6$
because it consists of 2545 correspondences between words and integer vectors.

\begin{theorem}
  The universal finite type invariant $\bar{\Gamma}_6$ classifies
  Gauss words of rank $5$ or less as shown in Table \ref{g6_classified_non_zero_words}.
  This classification is complete except for the distinction between the two sets of words;
  the classification of the two sets $\{ w_1, w_2 \}$ and $\{ w_3, w_4 \}$ remains unknown, where
  \begin{align*}
    \begin{array}{ll}
      w_1 = ABCDBEACED, & w_2 = ABCDECAEBD, \\
      w_3 = ABCADBECDE, & w_4 = ABCACDEDBE,
    \end{array}
    \label{g4_non_zero_words}
  \end{align*}
  and $w_1$ is homotopic to $w_2$ and $w_3$ is homotopic to $w_4$.
\end{theorem}

\section*{Acknowledgements}

The authors would like to thank Masamichi Kuroda for useful discussions.

\bibliographystyle{abbrv}
\bibliography{papers_en,books_en}

\begin{thebibliography}{10}

\bibitem{Bar-Natan2011}
D.~Bar-Natan, I.~Halacheva, L.~Leung, and F.~Roukema.
\newblock Some dimensions of spaces of finite type invariants of virtual knots.
\newblock {\em Exp. Math.}, 20(3):282--287, 2011.

\bibitem{Birman1993}
J.~S. Birman and X.-S. Lin.
\newblock Knot polynomials and {V}assiliev's invariants.
\newblock {\em Invent. Math.}, 111(2):225--270, 1993.

\bibitem{Gibson2011b}
A.~Gibson.
\newblock Homotopy invariants of {G}auss words.
\newblock {\em Math. Ann.}, 349(4):871--887, 2011.

\bibitem{Gibson2011}
A.~Gibson and N.~Ito.
\newblock Finite type invariants of nanowords and nanophrases.
\newblock {\em Topology Appl.}, 158(8):1050--1072, 2011.

\bibitem{Goussarov2000}
M.~Goussarov, M.~Polyak, and O.~Viro.
\newblock Finite-type invariants of classical and virtual knots.
\newblock {\em Topology}, 39(5):1045--1068, 2000.

\bibitem{Heggernes2006}
P.~Heggernes.
\newblock Minimal triangulations of graphs: a survey.
\newblock {\em Discrete Math.}, 306(3):297--317, 2006.

\bibitem{Kauffman1999}
L.~H. Kauffman.
\newblock Virtual knot theory.
\newblock {\em European J. Combin.}, 20(7):663--690, 1999.

\bibitem{Lang2002}
S.~Lang.
\newblock {\em Algebra}, volume 211 of {\em Graduate Texts in Mathematics}.
\newblock Springer-Verlag, New York, third edition, 2002.

\bibitem{Manturov2009a}
V.~O. Manturov.
\newblock On free knots.
\newblock {\em ArXiv e-prints}, Jan. 2009.

\bibitem{Manturov2010a}
V.~O. Manturov.
\newblock Parity in knot theory.
\newblock {\em Mat. Sb.}, 201(5):65--110, 2010.

\bibitem{Turaev2006}
V.~Turaev.
\newblock Knots and words.
\newblock {\em Int. Math. Res. Not.}, pages Art. ID 84098, 23, 2006.

\bibitem{Turaev2007}
V.~Turaev.
\newblock Lectures on topology of words.
\newblock {\em Jpn. J. Math.}, 2(1):1--39, 2007.

\bibitem{Turaev2007a}
V.~Turaev.
\newblock Topology of words.
\newblock {\em Proc. Lond. Math. Soc. (3)}, 95(2):360--412, 2007.

\bibitem{Vassiliev1992}
V.~A. Vassiliev.
\newblock {\em Complements of discriminants of smooth maps: topology and
  applications}, volume~98 of {\em Translations of Mathematical Monographs}.
\newblock American Mathematical Society, Providence, RI, 1992.
\newblock Translated from the Russian by B. Goldfarb.

\end{thebibliography}

\end{document}